\documentclass[12pt,a4paper]{article}
\usepackage{a4wide,amsmath,amsthm,epsfig,graphicx,psfrag}
\usepackage{amsmath,amsthm}
\usepackage{amsfonts}
\usepackage{amssymb}
\usepackage{enumerate}
\usepackage[latin1]{inputenc}  

\newcommand{\Ex}{\mathbb{E}}
\newcommand{\R}{\mathbb{R}}

\newcommand{\hY}{\hat{Y}}
\newcommand{\hX}{\hat{X}}
\newcommand{\kp}{\kappa}
\newcommand{\tPi}{\tilde{\Pi}}

\newtheorem{theorem}{Theorem} 
\newtheorem{proposition}[theorem]{Proposition}
\newtheorem{lemma}[theorem]{Lemma}

\newtheorem{remark}[theorem]{Remark}

\def\title#1{\hfil\break\hfil\break\hfil\break
\par\addvspace\baselineskip\noindent\baselinestretch
\ignorespaces{\LARGE\bf#1}\hfil\break}

\def\author#1{\par\addvspace\baselineskip\noindent
\ignorespaces{\large\bf#1}}

 \def\institute#1{\par\addvspace\baselineskip\noindent
 \ignorespaces{\small#1}\hfil\break}

\begin{document}

\title{Strong convergence of some drift implicit Euler scheme. Application to the CIR process. }
\author{Aur\'{e}lien Alfonsi}
\institute{CERMICS, projet MATHFI, Ecole Nationale des Ponts et Chauss\'{e}es, 6-8
avenue Blaise Pascal, Cit\'{e} Descartes, Champs sur Marne,
77455 Marne-la-vall\'{e}e, France. \\ e-mail : {\tt alfonsi@cermics.enpc.fr } \\
\today }
{\abstract We study the convergence of a drift implicit scheme for
  one-dimensional SDEs that was considered
  by Alfonsi~\cite{A} for the Cox-Ingersoll-Ross (CIR) process. Under 
  general conditions, we obtain a strong convergence of order~$1$. In the CIR case,
  Dereich, Neuenkirch and Szpruch~\cite{DNS}  have shown recently a strong convergence
  of order~$1/2$ for this scheme. Here, we obtain a strong convergence
  of order~$1$  under more restrictive assumptions on the CIR
  parameters. \\
  
}
{\noindent {\it Keywords: } \it Discretization scheme, 
  Cox-Ingersoll-Ross model, Strong error, Lamperti transformation.\\
  {\it AMS Classification (2010):} 65C30, 60H35 \\
}

%
%
%
\thispagestyle{empty}

This paper analyses the strong convergence error of a
discretization scheme for the Cox-Ingersoll-Ross (CIR) process and complements
a recent paper by~Dereich, Neuenkirch and Szpruch~\cite{DNS}. The CIR
process, which is widely used in financial modelling, follows the SDE:
\begin{equation}\label{SDE_CIR}
dX_t=(a-kX_t)dt +\sigma \sqrt{X_t} dW_t, \  X_0=x.
\end{equation}
Here, $W$ denotes a standard Brownian motion, $a\ge 0$, $k\in \R$, $\sigma >
0$ and $x\ge 0$.  This SDE has a unique strong solution that is
nonnegative. It is even positive when $\sigma^2\le 2a$ and $x>0$, which we
assume in this paper. It is
well-known that the usual Euler-Maruyama scheme is not defined
for~\eqref{SDE_CIR}. Different ad-hoc discretization schemes have thus been
proposed in the literature (see references in~\cite{DNS}). Here, we focus on a
drift implicit scheme that has been proposed in Alfonsi~\cite{A}. We consider
a time horizon~$T>0$ and a regular time grid:
$$t_k=\frac{kT}{n}, \ 0 \le k \le n. $$
By It\^o's formula, $Y_t=\sqrt{X_t}$ satisfies :
\begin{equation}\label{SDE_sqrtCIR}
dY_t=\left( \frac{a-\sigma^2/4}{2Y_t}-\frac{k}{2}Y_t \right) dt + \frac{\sigma}{2}  dW_t, \  Y_0=\sqrt{x}.
\end{equation}
We consider the following drift implicit Euler scheme
\begin{equation}\label{imp_sch_CIR}
\hY_0=\sqrt{x},\ \hY_t= \hY_{t_k}+\left(
  \frac{a-\sigma^2/4}{2\hY_{t}}-\frac{k}{2}\hY_{t} \right) (t-t_k) +
\frac{\sigma}{2}  (W_{t}-W_{t_k}), t\in (t_k,t_{k+1}].
\end{equation}
The equation~\eqref{imp_sch_CIR} is a quadratic equation that has a unique
positive equation:
{\small $$ \hY_t=\frac{\hY_{t_k}+\frac{\sigma}{2}(W_{t}-W_{t_k}) + \sqrt{\left(\hY_{t_k}+\frac{\sigma}{2}(W_{t}-W_{t_k})
      \right)^2 +2\left(1+ \frac{k}{2}(t-t_k)
      \right)\left(a-\frac{\sigma^2}{4} \right)(t-t_k) } }{2\left(1+
      \frac{k}{2}(t-t_k)  \right)},$$}

\noindent provided that the time-step is small enough ($T/n\le
2/\max(-k,0)$ with the convention $2/0=+\infty$). Last, we set $\hX_t=(\hY_t)^2$, $t\in (t_k,t_{k+1}]$. It is
shown in~\cite{A} that this scheme has uniformly bounded moments. We
recall now the main result of Dereich, Neuenkirch and Szpruch~\cite{DNS} that
gives a strong error convergence of order~$1/2$.
\begin{theorem}\label{thm_dns}
Let $x>0$, $2a>\sigma^2$ and $T>0$. Then, for all
$p\in[1,\frac{2a}{\sigma^2})$, there is a constant $K_p>0$ such that for any  $n \ge \frac{T}{2} \max(-k,0)$,
$$\left(\Ex \left[\max_{t\in[0,T]} |\hX_t-X_t|^p \right] \right)^{1/p}\le K_p \sqrt{\frac{T}{n}}.$$  
\end{theorem}
\noindent Let us remark that, contrary to~\cite{DNS}, we do not consider a linear interpolation between
 $t_k$ and~$t_{k+1}$ here for $\hX_t$. This removes the logarithm term
of Theorem~1.1 in~\cite{DNS}.

The strong convergence rate of~$\hX$ is studied numerically in
Alfonsi~(\cite{A}, Figure~2). This numerical study shows that the strong
convergence rate depends on the parameters $\sigma^2$ and $a$. When $\sigma^2/a$
is small enough, a strong convergence of order~$1$ is observed. The scope of the
paper is to prove the following result.
\begin{theorem}\label{main_thm}
Let $x>0$, $a>\sigma^2$ and $T>0$. Then, for all
$p\in[1,\frac{4a}{3\sigma^2})$, there is a constant $K_p>0$ such that for any  $n \ge \frac{T}{2} \max(-k,0)$,
$$\left(\Ex \left[\max_{t\in[0,T]} |\hX_t-X_t|^p \right] \right)^{1/p}\le K_p \frac{T}{n}.$$  
\end{theorem}
\noindent Thus, we get a strong convergence of order~$1$ under more restrictive
conditions on $\sigma^2/a$. Both theorems are complementary and are compatible
with the numerical study of~\cite{A}, which indicates that the strong
convergence order downgrades as long as $\sigma^2/a$ increases.

The paper is structured as follows. We first prove that $\left(\Ex
  \left[\max_{t\in[0,T]} |\hY_t-Y_t|^p \right] \right)^{1/p}\le K_p
\frac{T}{n}$ under a general framework for $Y$ and $\hY$ that extends~\eqref{SDE_sqrtCIR}
and~\eqref{imp_sch_CIR}. Then, we
deduce Theorem~\ref{main_thm} from this result. Also, we construct an
analogous drift implicit scheme for general one-dimensional diffusion, and get a strong convergence
of order one under suitable assumptions on the coefficients. This scheme has
the advantage to be
naturally defined in the diffusion domain like $\R_+^*$ for the CIR case. 

\subsection*{A general framework for~$Y$ and~$\hY$}

Let $c\in [-\infty,+\infty)$, $I=(c,+\infty)$ and $d\in I$. We consider in
this section the following SDE defined on $I=(c,+\infty)$:  
\begin{equation} \label{def_Y}
  dY_t=f(Y_t) dt +\gamma dW_t,\ t\ge 0, Y_0=y\in I,
\end{equation}
with $\gamma>0$. We make the following monotonicity assumption on the drift coefficient~$f$:
\begin{equation}\label{hyp_f}
f:I \rightarrow \R \text{ is } \mathcal{C}^2, \text{ such that }
\exists \kp \in \R, \forall y,y' \in I,\ y \le y', \ f(y')-f(y)\le \kp (y'-y).
\end{equation}
Besides, we assume 
 \begin{equation} \label{cond_I} v(x)=\int_d^x \int_d^y \exp \left(-\frac{2}{\gamma^2} \int_z^y f(\xi)d\xi \right)
 dzdy \text{ satisfies } \lim_{x\rightarrow c+}v(x)=-\infty.
\end{equation}
The Feller's test (see e.g. Theorem~5.29 p.~348 in~\cite{KS}) ensures that $Y$ never
reaches~$c$ nor $+\infty$ by~\eqref{hyp_f}, and the SDE~\eqref{def_Y} admits a unique strong solution
 on~$I$.
 
Let us now define the drift implicit scheme. Let us first observe that for
$h>0$ such that $\kp h<1$,
$y\mapsto y-h f(y)$ is a bijection from $I$ to~$\R$. Indeed, it is continuous and we
have from~\eqref{def_Y}:
$$y\le y', \  y'-y-h(f(y')-f(y))\ge (1-\kp h)(y'-y).$$
This shows the claim for $c=-\infty$. For $c>-\infty$, we first remark that
$\lim_{c+}f $ exists from~\eqref{hyp_f}, and is necessarily equal to~$+\infty$
from~\eqref{cond_I}. Thus, for $n$ such that
$\kp T/n< 1$, the following drift implicit Euler scheme is well defined
\begin{equation}\label{gen_sch_Y}\hY_0=y, \ \hY_t=\hY_{t_k}+f(\hY_t)(t-t_k)+\gamma (W_t-W_{t_k}), \ t\in
(t_k,t_{k+1}], \ 0\le k\le n-1,
\end{equation}
and satisfies $\hY_t \in I$, for any $t\in [0,T]$. From a computational point
of view, let us remark here that in
cases where $\hY_{t_{k+1}}$ cannot be solved explicitly like in the CIR case,
 $\hY_{t_{k+1}}$ can still be quickly computed from~$\hY_{t_k}$ and $W_{t_{k+1}}-W_{t_k}$ thanks to the monotonicity of~$y\mapsto
y-(T/n) f(y)$ by using for example a dichotomic search.

The drift implicit Euler scheme (also known as backward Euler scheme) has been
studied by Higham, Mao and Stuart~\cite{HMS} for SDEs on~$\R^d$ with a
Lipschitz condition on the diffusion coefficient and a monotonicity condition
on the drift coefficient that extends~\eqref{hyp_f}. They show a strong
convergence of order~$1/2$ in this general setting.

\begin{proposition}\label{main_prop} Let $p\ge 1$ and $n> 2 \kp T$. Let us assume that
  \begin{equation}\label{Hyp_cv}\Ex \left[\left(\int_0^T |f'(Y_u)f(Y_u)+\frac{\gamma^2}{2}f''(Y_u)|du
    \right)^p \right]<\infty \text{ and } \Ex \left[\left(\int_0^T (f'(Y_u))^2du
    \right)^{p/2} \right]<\infty.
\end{equation}
Then, there is a constant $K_p>0$ such that:
  $$\left(\Ex \left[ \max_{t\in [0,T]} |\hY_t-Y_t|^p \right]\right)^{1/p} \le
  K_p \frac{T}{n}. $$
\end{proposition}
Before proving this result, let us recall that the same result holds for the
usual (drift explicit) Euler-Maruyama scheme when $I=\R$ (i.e. $c=-\infty$), under some regularity
assumption on~$f$. Said differently, the Euler-Maruyama scheme
($\bar{Y}_{t_{k+1}}=\bar{Y}_{t_{k}}+f(\bar{Y}_{t_{k}})T/n+\gamma(W_{t_{k+1}}-W_{t_{k}})$)
coincides with
the Milstein scheme when the diffusion coefficient is constant, and its order
of strong convergence is thus equal to one. The main advantage of the
drift implicit scheme is that it is well defined when $c>-\infty$ while the
Euler-Maruyama is not, since the Brownian increment may lead outside~$I$.

\begin{proof}
We may assume without loss of generality that $\kp \ge 0$. For $t\in [0,T]$,
we set $e_t=\hY_t-Y_t$. From~\eqref{hyp_f}, there is $\beta_t \le \kp$, such
that $f(\hY_t)-f(Y_t)=\beta_t e_t$. For $0\le k\le n-1$, we have
$$e_{t_{k+1}}=e_{t_k}+[f(\hY_{t_{k+1}})-f(Y_{t_{k+1}})]\frac{T}{n}+\int_{t_k}^{t_{k+1}}f(Y_{t_{k+1}})-f(Y_s)ds, $$
and then, by using It\^o's formula:
\begin{equation}\label{eq_intermed} \left(1-\beta_{t_{k+1}}\frac{T}{n}\right)e_{t_{k+1}}=
e_{t_k}+\int_{t_k}^{t_{k+1}}(u-t_k)[f'(Y_u)f(Y_u)+\frac{\gamma^2}{2}f''(Y_u)]du
+\gamma \int_{t_k}^{t_{k+1}}(u-t_k)f'(Y_u)dW_u.
\end{equation}
For $u\in[0,T]$, we denote by $\eta(u)$ the integer such
  that $t_{\eta(u)}\le u <t_{\eta(u)+1}$. We set  
$\Pi_0=1$, $\Pi_k= \prod_{l=1}^k (1-\beta_{t_l} \frac{T}{n}
  )$, $\tilde{e}_k=\Pi_k e_{t_k}$, $\tilde{\Pi}_k=\Pi_k/(1-\kp T/n)^k$ and
  $$M_{t}=\int_{0}^{t}(1-\kp T/n)^{\eta(u)}(u-t_{\eta(u)})\gamma
  f'(Y_u)dW_u.$$
Let us remark that $\Pi_k>0$, $\tPi_k\ge 1$ and $\tPi_k$ is nondecreasing with
respect to~$k$.  By multiplying equation~\eqref{eq_intermed} by~$\Pi_k$, we get
$$\tilde{e}_{k+1}=\tilde{e}_{k}+\Pi_k\left(\int_{t_k}^{t_{k+1}}
(u-t_k)[f'(Y_u)f(Y_u)+\frac{\gamma^2}{2}f''(Y_u)]du + \int_{t_k}^{t_{k+1}}
(u-t_k)\gamma f'(Y_u)dW_u\right).$$
Then, we obtain $\tilde{e}_{k}= \int_0^{t_k} \Pi_{\eta(u)}
(u-t_{\eta(u)})[f'(Y_u)f(Y_u)+\frac{\gamma^2}{2}f''(Y_u)]du + \sum_{l=0}^{k-1} \tPi_l
(M_{t_{l+1}}-M_{t_l}) $ by summing over~$k$ and finally get
\begin{equation}\label{intermed_ek}
  e_{t_k}= \int_0^{t_k} \frac{\Pi_{\eta(u)}}{\Pi_k}
(u-t_{\eta(u)})[f'(Y_u)f(Y_u)+\frac{\gamma^2}{2}f''(Y_u)]du + \frac{1}{\Pi_k} \sum_{l=0}^{k-1} \tPi_l
(M_{t_{l+1}}-M_{t_l}).
\end{equation}
Since $\frac1{1-x}\le \exp(2x)$ for $x\in[0,1/2]$, we have
$$0\le l\le k\le n,\  0< \frac{\Pi_l}{\Pi_k}=\frac{1}{(1-\kp T/n)^{k-l}}
\frac{\tPi_l}{\tPi_k}\le \exp \left(2(k-l)\kp \frac{T}{n}\right) \le \exp(2\kp
T) .$$
On the other hand, an Abel transformation gives $\sum_{l=0}^{k-1} \tPi_l
(M_{t_{l+1}}-M_{t_l})=\tPi_{k-1} M_{t_k}
+\sum_{l=1}^{k-1}(\tPi_{l-1}-\tPi_l)M_{t_l}$ and thus 
$$\left|\sum_{l=0}^{k-1} \tPi_l
(M_{t_{l+1}}-M_{t_l})\right|\le \tPi_{k-1}|M_{t_k}|+\sum_{l=1}^{k-1}(\tPi_{l}-\tPi_{l-1})
|M_{t_l}|\le 2 \tPi_{k} \max_{1\le l \le k} |M_{t_k}|,$$
since $\tPi_k$ is nondecreasing. From~\eqref{intermed_ek} and $\frac{\tPi_k}{\Pi_k}=\frac{1}{(1-\kp T/n)^k}\le \exp(2\kp
T)$, we get
$$|e_{t_k}|\le \exp(2 \kp T) \left( \frac{T}{n} \int_0^{t_k}
  |f'(Y_u)f(Y_u)+\frac{\gamma^2}{2}f''(Y_u)|du+ 2 \max_{0\le l\le k} |M_{t_l}|\right). $$
Since the right hand side is nondecreasing with respect to~$k$, we can replace
the left hand side by~$\max_{0\le l\le k} |e_{t_l}|$. Burkholder-Davis-Gundy
inequality gives that $$\Ex \left[\max_{0\le l\le n} |M_{t_l}|^p\right]\le C_p \gamma^p (T/n)^p \Ex \left[\left(\int_0^T (f'(Y_u))^2du
    \right)^{p/2} \right],$$
  since $0\le (1-\kp T/n)^{\eta(u)} \le 1$.
Thus, there is a positive
constant~$K$ depending on~$\kp$, $T$ and $p$ such that:
\begin{align}\label{max_discret}
\Ex\left[ \max_{0\le l \le n} |e_{t_l}|^p \right]
\le K \left(\frac{T}{n}\right)^p  & \left(  \Ex \left[\left( \int_0^{T}
  |f'(Y_u)f(Y_u)+\frac{\gamma^2}{2}f''(Y_u)|du\right)^p \right] \right.  \\& \left.+ \gamma^p \Ex \left[\left(\int_0^T (f'(Y_u))^2du
    \right)^{p/2} \right] \right) \nonumber
\end{align}
It remains to show the analogous upper bound for $\Ex[\max_{t\in [0,T]}
|e_t|^p]$. Similarly to~\eqref{eq_intermed}, we have for $t\in[t_k,t_{k+1}]$:
$$\left(1-\beta_t (t-t_k) \right)e_t=
e_{t_k}+\int_{t_k}^{t}(u-t_k)[f'(Y_u)f(Y_u)+\frac{\gamma^2}{2}f''(Y_u)]du
+\gamma \int_{t_k}^{t}(u-t_k)f'(Y_u)dW_u.$$
Since $\left(1-\beta_t (t-t_k) \right) \ge 1/2$, we get:
\begin{align*}
  \max_{t\in [t_k,t_{k+1}]} |e_t| \le 2 & \left(
  |e_{t_k}|+\frac{T}{n}\int_{t_k}^{t_{k+1}}|f'(Y_u)f(Y_u)+\frac{\gamma^2}{2}f''(Y_u)|du
\right. \\
& \left.+ \gamma \max_{t\in [t_k,t_{k+1}]}  \left|\int_{t_k}^{t}(u-t_k)f'(Y_u)dW_u\right| \right),
\end{align*}
and thus
\begin{align*}\max_{t\in [0,T]} |e_t|^p \le 2^p3^{p-1} & \left( \max_{0\le k \le n} |e_{t_k}|^p
 + \left(\frac{T}{n}\right)^p
 \left(\int_{0}^{T}|f'(Y_u)f(Y_u)+\frac{\gamma^2}{2}f''(Y_u)|du \right)^p
\right.\\ &\left.
+ \gamma^p \max_{0\le s\le t \le T}  \left|\int_{s}^{t}(u-t_{\eta(u)})f'(Y_u)dW_u\right|^p \right).
\end{align*}
Since $\left|\int_{s}^{t}(u-t_{\eta(u)})f'(Y_u)dW_u\right|^p\le2^p\left(
\left|\int_{0}^{t}(u-t_{\eta(u)})f'(Y_u)dW_u\right|^p+
\left|\int_{0}^{s}(u-t_{\eta(u)})f'(Y_u)dW_u \right|^p\right)$,
we conclude by using once again Burkholder-Davis-Gundy
inequality,~\eqref{max_discret} and~\eqref{Hyp_cv}.
\end{proof}


\subsection*{Application to the CIR process}
For the CIR case, we have $c=0$ (i.e. $I=\R_+^*$), $f(y)=\frac{a-\sigma^2/4}{2y}-\frac{k}{2}y$ and
$\gamma=\sigma/2$. When $2a \ge
\sigma^2$, we can check that both~\eqref{hyp_f} and~\eqref{cond_I} are
satisfied. By Jensen inequality, \eqref{Hyp_cv} holds if we have
\begin{equation}\label{Hyp_cv2}
  \int_0^T \Ex[|f'(Y_u)f(Y_u)|^p+|f''(Y_u)|^p+|f'(Y_u)|^{2\vee p}]du<\infty.
\end{equation}
The moments of the CIR process can be uniformly bounded on~$[0,T]$ under the following condition (see~\cite{DNS} equation~(7)):
\begin{equation}\label{moments_cir}\sup_{t\in[0,T]} \Ex[X_t^q]< \infty \text{ for } q>-\frac{ 2a}{\sigma^2}.
\end{equation}
Condition~\eqref{Hyp_cv2} will hold as soon as $\sup_{t\in[0,T]}
\Ex[Y_t^{-(4\vee 3p)}]
=\sup_{t\in[0,T]} \Ex[X_t^{-(2\vee \frac{3}{2}p)}]<\infty$. This is satisfied when
$\sigma^2<a$ and $p<\frac{4}{3}\frac{a}{\sigma^2}$, and we have $\left(\Ex \left[ \max_{t\in [0,T]} |\hY_t-Y_t|^p \right]\right)^{1/p} \le
  K_p \frac{T}{n}. $

From now on, we
assume that $\sigma^2<a$ and consider $1\le p<\frac{4}{3}\frac{a}{\sigma^2}$. Let $\varepsilon>0$
such that $p(1+\varepsilon)<\frac{4}{3}\frac{a}{\sigma^2}$. Since $\hX_t-X_t=(\hY_t-Y_t)(\hY_t+Y_t)$,
we have by H\"older's inequality:
$$ \Ex \left[\max_{t\in [0,T]} |\hX_t-X_t|^p \right]^{\frac{1}{p}}\le \Ex \left[\max_{t\in [0,T]}
|\hY_t-Y_t|^{p(1+\varepsilon)}\right]^\frac{1}{p(1+\varepsilon)} \Ex\left[\max_{t\in [0,T]}
|\hY_t+Y_t|^{p\frac{1+\varepsilon}{\varepsilon}}\right]^{\frac{\varepsilon}{p(1+\varepsilon)}}.$$
The moment boundedness of $\hat{Y}$ is checked in~\cite{A} and~\cite{DNS}, and
the second expectation is thus finite. Proposition~\ref{main_prop} gives Theorem~\ref{main_thm}.

\subsection*{Application to $dX_t=(a-kX_t)dt+\sigma X_t^\alpha dW_t$, with
  $1/2<\alpha<1$}

We consider this SDE starting from $X_0=x>0$ 
with parameters $a>0$, $k\in \R$ and $\sigma >0$. This SDE is known to have
a unique strong positive solution~$X$, which can be checked easily by Feller's
test for explosions. We set $$Y_t=X_t^{1-\alpha}.$$ It is
defined on $I=\R_+^*$ and satisfies~\eqref{def_Y} with
$$f(y)=(1-\alpha)\left(a
  y^{-\frac{\alpha}{1-\alpha}}-ky-\alpha\frac{\sigma^2}{2}y^{-1} \right)\text{ with
} \gamma=\sigma (1-\alpha).$$
Since $a>0$ and $\frac{\alpha}{1-\alpha}>1$, $f$ is decreasing on
$(0,\varepsilon)$, for $\varepsilon>0$ small enough. It is also clearly
Lipschitz on $[\varepsilon,+\infty)$, and~\eqref{hyp_f} is thus
satisfied. Also, we check easily that~\eqref{cond_I} holds. The drift implicit
scheme $(\hat{Y}_t,t\in[0,T])$ given by~\eqref{gen_sch_Y} is thus well defined
for large~$n$ and we set:
$$\hat{X}_t=(\hat{Y}_t)^{\frac{1}{1-\alpha}}.$$ To apply
Proposition~\ref{main_prop}, it is enough to check that~\eqref{Hyp_cv2} holds. To do so, we
have the following lemma.
\begin{lemma}We have: $\forall q \in \R, \sup_{t\in[0,T]} \Ex[X_t^q]< \infty$.
\end{lemma}
\begin{proof}
For $q\ge 0$, it is well known that we even have $\Ex[ \max_{t\in[0,T]} X_t^q]<\infty$
from the sublinear growth of the SDE coefficients (see e.g. Karatzas and Shreve~\cite{KS}, p 306). Let
$q<0$. We set $Z_t=X_t^{2(1-\alpha)}$ and have:
$$ dZ_t=b(Z_t)dt + 2(1-\alpha)\sigma \sqrt{Z_t}dW_t, \text{ with } b(z)=2(1-\alpha)\left[a z^{\frac{1-2\alpha}{2(1-\alpha)}}-kz+\sigma^2\left(\frac{1}{2}-\alpha\right)
\right].$$
Since $\lim_{z\rightarrow 0^+}b(z)=+\infty$ and $b$ is Lipschitz on
$[\varepsilon,+\infty)$ for any $\varepsilon>0$, we can find for any~$M>0$ a
constant $k_M \in \R$ such that $b(z)\ge M -k_Mz$ for all $z>0$. We consider
then the following CIR process:
$$d\xi^M_t=(M-k_M \xi^M_t)dt+2(1-\alpha)\sigma \sqrt{\xi^M_t}dW_t, \ \xi^M_0=x^{2(1-\alpha)}.$$
From a comparison theorem (Proposition 2.18, p 293 in~\cite{KS}) we get
that $\forall t\ge 0, Z_t\ge \xi^M_t$ and thus $\sup_{t\in[0,T]} \Ex[Z_t^q]
\le \sup_{t\in[0,T]} \Ex[(\xi^M_t)^q]$. We conclude by
using~\eqref{moments_cir} and taking $M$ is arbitrary large.
\end{proof}
We can then apply Proposition~\ref{main_prop} and get, for any $p\ge 1$ and $n$
large enough, the existence of a constant
$K_p>0$ such that $\left(\Ex \left[ \max_{t\in [0,T]} |\hY_t-Y_t|^p \right]\right)^{1/p} \le
  K_p \frac{T}{n}. $ In particular, we get $\Ex[\max_{t\in[0,T]}
  \hY_t^p]<\infty$. We have $\hX_t=(\hY_t)^{\frac{1}{1-\alpha}}$ and  
  $$ \hat{y},y>0, \
  |\hat{y}^{\frac{1}{1-\alpha}}-y^{\frac{1}{1-\alpha}}|=\frac{1}{1-\alpha} \left|\int_y^{\hat{y}}
  z^{\frac{\alpha}{1-\alpha}}dz \right| \le\frac{1}{1-\alpha}|\hat{y}-y|
(\hat{y} \vee y)^{\frac{\alpha}{1-\alpha}}. $$
The Cauchy-Schwarz inequality leads then to
 $$\Ex \left[\max_{t\in [0,T]} |\hX_t-X_t|^p \right]^{\frac{1}{p}} \le
 \frac{1}{1-\alpha} \Ex \left[\max_{t\in [0,T]} |\hY_t-Y_t|^{2p}
 \right]^{\frac{1}{2p}}\Ex \left[\max_{t\in [0,T]} (\hY_t\vee Y_t)^{\frac{2p\alpha}{1-\alpha}}
 \right]^{\frac{1}{2p}} \le \tilde{K}_p\frac{T}{n}. $$

\subsection*{Strong convergence towards~$X$ in a general framework}
Let us now consider a one-dimensional SDE with Lipschitz
coefficients~$b,\sigma:\R\rightarrow \R$:
$$ dX_t=b(X_t)dt+\sigma(X_t)dW_t, \ X_0=x.$$
We will consider the Lamperti transformation of this SDE. 
 We assume that there exist $0<\underline{\sigma}<\overline{\sigma}$ such that
$\underline{\sigma} \le \sigma(x)\le\overline{\sigma}$, so that
$$\varphi(x)=\int_0^x \frac{1}{\sigma(z)}dz \text{ is bijective on }\R,$$
Lipschitz and such that $\varphi^{-1}$ is Lipschitz. Besides, we assume that
$\sigma \in
\mathcal{C}^1$ and that $f= \left(\frac{b}{\sigma}-\frac{\sigma'}{2} \right)
\circ \varphi^{-1}$ satisfies~\eqref{hyp_f},~\eqref{cond_I} and:
$$\exists K>0, q>0, \forall y \in \R, \ |f'(y)|+|f''(y)|\le C(1+|y|^q).$$
Then $Y_t=\varphi(X_t)$ satisfies $dY_t=f(Y_t)dt+  dW_t$.
The Lipschitz assumption on the coefficients~$b$ and $\sigma$ ensures the
boundedness of moments of~$X$ and thus of~$Y$. The condition~\eqref{Hyp_cv} is thus satisfied and the conclusion of Proposition~\ref{main_prop}
holds. Then, defining $\hY$ by~\eqref{gen_sch_Y} and setting
$\hX_t=\varphi^{-1}(\hY_t)$ for $t\in [0,T]$, we get that:
$$\exists K_p>0,\ \left(\Ex \left[ \max_{t\in [0,T]} |\hX_t-X_t|^p \right]\right)^{1/p} \le
  K_p \frac{T}{n}. $$
Let us mention that the same result holds under suitable conditions on~$f$ for the scheme
$\bar{X}_t=\varphi^{-1}(\bar{Y}_t)$, where $\bar{Y}$ denotes the Euler-Maruyama scheme
$d\bar{Y}_t=f(\bar{Y}_{t_{\eta(t)}})dt+dW_t$. The weak convergence of this
scheme has been studied by Detemple, Garcia and Rindisbacher~\cite{DGR}. 

\begin{remark}Let $\gamma>0$, $\varphi_\gamma(x)=\gamma \varphi(x)$ and
$f_\gamma(y)=\gamma f(y/\gamma)$. Then, $Y'_t=\varphi_\gamma(X_t)$ solves
$dY'_t= f_\gamma(Y'_t)dt + \gamma dW_t$. The associated drift implicit scheme
$$\hY'_0=\varphi_\gamma(X_0), \ \hY'_t=\hY'_{t_k}+f_\gamma(\hY'_t)(t-t_k)+\gamma (W_t-W_{t_k}), \ t\in
(t_k,t_{k+1}], \ 0\le k\le n-1,$$
clearly satisfies $ \hY'_t=\gamma \hY_t$. Thus,
$\hX_t=\varphi_\gamma^{-1}(\hY'_t)$: the scheme $\hX$ is unchanged when the
transformation
between~$X$ and~$Y$  is multiplied by a positive constant. 
\end{remark}
\noindent {{\bf Acknowledgments.} The author acknowledges the support of the Eurostars E!5144-TFM
project and of the ``Chaire Risques Financiers'' of Fondation du Risque.\\}

\end{document}